\documentclass[12pt]{article}
\usepackage{amsmath}
\usepackage{amsthm}
\usepackage{amscd}
\usepackage{amssymb,amsfonts}
\usepackage[alphabetic,nobysame]{amsrefs}
\usepackage{xcolor}
\usepackage{tikz}
\usepackage{tikz-cd}

\usepackage{hyperref}
\usepackage{latexsym}
\usepackage{todonotes}
\usepackage{nicefrac}
\usepackage{epsfig}
\usepackage{stmaryrd}
\usepackage{setspace}
\usepackage{enumerate}
\usepackage[all]{xypic}
\usepackage{bbm,ifpdf,tikz}
\usepackage{verbatim}
\ifpdf
\usepackage{pdfsync}
\fi

\newtheorem{theorem}{Theorem}[section]
\newtheorem*{theorem*}{Theorem}

\newtheorem{proposition}[theorem]{Proposition}
\newtheorem{corollary}[theorem]{Corollary}

\newtheorem{thmab}{Theorem}

\theoremstyle{definition}
\newtheorem{definition}[theorem]{Definition}
\newtheorem{example}[theorem]{Example}

\newtheorem{remark}[theorem]{Remark}

\newcommand{\excise}[1]{}

\newcommand{\Tor}{\operatorname{Tor}}

\renewcommand{\dim}{\operatorname{dim}}

\renewcommand{\and}{\qquad\text{and}\qquad}

\newcommand{\Hom}{\operatorname{Hom}}

\newcommand{\into}{\hookrightarrow}

\newcommand{\Z}{\mathbb{Z}}
\newcommand{\Q}{\mathbb{Q}}

\newcommand{\C}{\mathbb{C}}
\newcommand{\R}{\mathbb{R}}

\newcommand{\F}{\mathbb{F}}

\newcommand{\FI}{\operatorname{FI}}
\newcommand{\OI}{\operatorname{OI}}

\newcommand{\Aut}{\operatorname{Aut}}

\DeclareMathOperator{\Co}{\mathsf{Co}} 
\DeclareMathOperator{\Cot}{\mathsf{Cot}} 

\newcommand{\PCot}{\operatorname{\mathsf{PCot}}}
\newcommand{\cM}{\mathcal{M}}
\DeclareMathOperator{\Mod}{\mathsf{Mod}}

\DeclareMathOperator{\Set}{\mathsf{Set}}
\DeclareMathOperator{\Vect}{\mathsf{Vect}}
\newcommand{\tea}[1]{\textcolor{violet}{[Tea: #1]}} 

\title{Universality in the algebra and topology of cographs}
\author{Adityo Mamun, Jonathan Nalikka, and Eric Ramos}

\begin{document}

\maketitle

\begin{quote}
    \centering{In Memory of Jonathan Nalikka}
\end{quote}

{\small
\begin{quote}
\noindent {\em Abstract.}
    A finite simple graph $G$ is called a cograph if it does not contain the path on four vertices $P_4$ as an induced subgraph. It is classically known that the family of cographs are well-quasi-ordered by the induced subgraph relation \cite{D}. In preceding work of Knudsen and the third author \cite[Theorem 7.2]{KR}, it was shown that this well-quasi-order statement admitted a categorification, which allowed those authors to prove universal finite generation statements about the homology groups of configuration spaces on cographs \cite[Theorem 1.5]{KR}. In this work, we expand \cite[Theorem 7.2]{KR} to be compatible with the family of polynomial rings on the vertex sets of cographs. By consequence, we are able to prove a number of universality results related with edge and toric ideals of these polynomial rings. We also conclude strong restrictions on the kinds of topologies that can arise from graph complexes and anchored configuration spaces associated to cographs, as well as combinatorial constraints on the possible combinatorics of hyperplane arrangements of cographs. 
\end{quote} }

\section{Introduction}

A \textbf{cograph} is any graph not containing the path $P_4$ as an induced subgraph. Importantly, this class has a recursive characterization as well, being the smallest collection of graphs that contains the single vertex graph and is closed under edge complementation and disjoint union. This class has been the subject of study since at least the 1970's, because of a number of extremely useful properties. For instance, on the combinatorial side, invariants that are famously difficult to compute for general graphs are often times easier to compute for cographs. On the structural graph theory side, it was proven in the 1990's that cographs are well-quasi-ordered by the induced subgraph relation \cite{D}. More recently, Kahle used the recursive nature of cographs to prove that the regularity of their binomial edge ideals (see Definition \ref{def:BinIdeal}) could be cleanly bounded only in terms of the number of vertices of the graph \cite{kahle2019binomial}. The purpose of this manuscript is to unify and expand upon a number of these results from the literature via the recent and very active study of the representation theory of combinatorial categories \cite{sam2017grobner}.

Give a category $\mathcal{C}$, a representation of $\mathcal{C}$ (over a ring $k$) can be thought of as a functor from $\mathcal{C}$ to the category of finitely generated $k$-modules (See Definition \ref{def:Catrep}). More concretely, a representation fo $\mathcal{C}$ can be thought of as a collection of finitely generated $k$-modules, $M(A)$, one for every object $A$ of $\mathcal{C}$, such that every morphism $A \rightarrow B$ of $\mathcal{C}$ induces natural homomorphisms $M(A) \rightarrow M(B)$. We say that $M$ is \textbf{finitely generated} if there is a finite collection of objects $\{A_i\}$ in $\mathcal{C}$ such that for any object $A$ of $\mathcal{C}$, the module $M(A)$ is generated by the images of the $M(A_i)$ according to the aforementioned induced maps coming from the category.

For instance, if $G$ is any group, one can view $G$ as the category $\mathcal{C}_G$, with a single object $\star$ and with morphisms in bijection with the elements of $G$. Composition of these morphisms are then defined by the group law of $G$. In this specific case, a representation of $\mathcal{C}_G$ is then a finitely generated module over the group ring $k[G]$. Slightly more generally, any finite quiver $Q$ can be thought of as a category $\mathcal{C}_Q$ by having edges correspond to morphisms and vertices to objects. In this case, representations of $\mathcal{C}_Q$ become quiver representations of the quiver $Q$. In this work, we will largely be concerned with the category $\Co$, whose objects are cographs and whose morphisms are full embeddings between cographs. That is to say, the morphisms are injective graph homomorphisms whose images are induced subgraphs of the larger graph.

In a prior work \cite{KR}, the following critical technical theorem was proven.

\begin{theorem*}
    Let $M$ denote a representation of the category $\Co$ over a Noetherian ring $k$. If $M$ is finitely generated, then all of its submodules are also finitely generated.
\end{theorem*}

This Noetherianity result can be thought of as a kind of categorical version of the aformentioned fact that cographs are well-quasi-ordered by the induced subgraph relation (See \cite{sam2017grobner} for the precise relationship between the two or \cite{ramos2022graph} for a more expository take). It was shown in \cite{KR} how this theorem could then be used to prove various stability and universality phenomena in the homology groups of graph configuration spaces. For us, however, this result will not be sufficient! Indeed, if it is our hope to somehow unify both the structural graph theory of cographs, and the algebra results of Kahle \cite{kahle2019binomial}, we need a higher categorification.

Define the functor $A_{|\bullet|}:\Co \rightarrow k-\text{Alg}$ by the obvious extension of the assignments $A_{|G|}= k[x_v \mid v \in V_G]$. We refer to this as the \textbf{polynomial ring over $\Co$}. Then an $A_{|\bullet|}$-module can be defined as a $\Co$-module $M$ such that for each cograph $G$, $M(G)$ admits an action by $A_{|G|}$, in such a way that all of the natural diagrams commute (see Definition \ref{def:cographPolyRing}) For instance, the edge ideals associated to cographs form an $A_{|\bullet|}$-module. Just as with representations of $\Co$, one can define the notion of a finitely generated $A_{|\bullet|}$-module, leading to the following main technical result of this work

\begin{thmab}\label{thm:MainTech}
     Let $M$ denote a module over the cograph polynomial ring over a Noetherian ring $k$. If $M$ is finitely generated, then all of its submodules are also finitely generated.
\end{thmab}

\begin{remark}
    The binomial edge ideals of \cite{kahle2019binomial} are not ideals of the standard polynomial ring, but rather the polynomial ring $\Q[x_v,y_w \mid v,w \in V_G]$. Theorem \ref{thm:MainTech} has a generalization for this context as well, which can be found in Theorem \ref{generalizedMain}.
\end{remark}

After proving this result, our next goal will be to apply it in a number of contexts. To start, we look back to the regularity results of \cite{kahle2019binomial}. For each cograph $G$, let $I_G$ denote an ideal of the polynomial ring $A_{|G|}$. We say that a family of such ideals $\{I_G\}$ is \textbf{hereditary} if whenever $G$ is an induced subgraph of $G'$, the inclusion of vertex sets $V_G \hookrightarrow V_{G'}$ induces an inclusion $I_G \hookrightarrow I_{G'}$. In other words, the collection $\{I_G\}$ is hereditary whenever $I_\bullet$ is a submodule of $A_{|\bullet|}$, thought of as an $A_{|\bullet|}$-module. As alluded to throughout this introduction, the most classic example of a hereditary family of ideals is the family of edge ideals $I_G = (x_vx_w \mid \{v,w\} \in E_G)$.

\begin{thmab}
    If $\{I_G\}$ is a hereditary family of ideals, then there exists a finite list of cographs $G_i$ such that if $G$ is any cograph, the ideal $I_G$ is generated by the images of the generators of the ideals $I_{G_i}$ under the maps induced by the induced subgraph relation. In particular, if $k$ is a field, then for any $q \geq 0$ there exists a finite list of cographs $\{G_{i,q}\}$ such that for any cograph $G$, the $q$-th syzygies of $I_G$ are generated by the $q$-th syzygies of the $I_{G_{i,q}}$.
\end{thmab}

One important consequence of this theorem is that we can prove that the \textbf{Betti numbers} of hereditary families of ideals are universally bounded only in terms of the index $q$. See Corollary \ref{ComAlgCor} for a precise statement along these lines.

Expanding upon the commutative algebra that partially inspired this work, our next applications are all in the various realms of topology. Let $G$ be a graph with vertex set $V$. A \textbf{vertex monotone graph complex} is a functor $\Delta_\bullet$ from the category of graphs with full embeddings to the category of abstract simplicial complexes and simplicial maps, with the added requirement that the vertex set of $\Delta_G$ is (a subset of) the vertex set of $G$, for all graphs $G$. The first and arguably most natural example of this kind of complex is the \textbf{independence complex} $\mathcal{I}_G$ of the graph $G$. This is the simplicial complex whose simplicies are independent sets of vertices. It is a simple exercise in combinatorial topology to see that the disjoint union of two graphs has an independence complex that is the topological join of the two individual independence complexes, while the independence complex of the join of two graphs is the disjoint union of the individual independence complexes. In particular, the topology of the independence complex of a cograph can be recursively determined due to the aforementioned recursive characterization of cographs. Our first major topological result will say that, at least up to homology, this finite determination will always hold for vertex monotone graph complexes of cographs.

\begin{thmab}
    Let $\Delta_\bullet$ denote a vertex monotone graph complex. Then for any fixed $i$, the $\Co$-module
    \[
    G \mapsto H_i(\Delta_G;\Z)
    \]
    is finitely generated. Moreover, for any field $k$, vertex monotone graph complex $\Delta_\bullet$ and any $i \geq 0$, there are at most finitely many cographs for which $\dim \tilde{H}_{|V_G| - i - 2}(\Delta_G ; k)$ is non-zero
\end{thmab}

We will then conclude our topological applications by extending a number of directions from the precursor work \cite{KR}. As stated above, one of the primary results of that work was to prove statements about the homology groups of configuration spaces of cographs. In this work we will consider variations on this particular theme.

For instance, For any pair $(G,K)$ (called a \textbf{pointed graph}) of a graph with a selection of vertices, Kozlov \cite{kozlov2023stirling,kozlov2022configuration} considers the subcomplex of the cubical complex $G^n$, of all $n$-tuples of points in $G$ for which every vertex in $K$ appears at least once. These so-called \textbf{anchored configuration spaces} have been shown to have applications to a variety of problems arising from resource management and logistics. In his works \cite{kozlov2022configuration,kozlov2023stirling,kozlov2024homology}, Kozlov was able to compute the homotopy types of these spaces for all trees, as well as study their Euler characteristics. A number of computations were also completed for the case of the cycle graph. One is left with two major observations when studying these results. Firstly, virtually all of the formulas and computations done are seen to depend on very little actual topological structure of the underlying graph. The majority of the topology can often be seen to reduce to purely combinatorial considerations. Secondly, the number of computations that have been thus far completed fall into only a small number of families of graphs. The following theorem puts extremely heavy restrictions on the homology groups of these spaces across the entirety of cographs.

\begin{thmab}
    Let $i,r,n \geq 0$ be fixed integers. Then,
        \begin{enumerate}
            \item For any $i,r,n \geq 0$, there exists a finite list of pointed cographs $\{(G_j,K_j)\}_{j = 1}^l$ with $|K_j| = r$, such that for any pointed cograph $(G,K)$ with $|K| = r$, the homology group $H_i(\Sigma(G,K,n);\Z)$ is generated by the homology classes of $H_i(\Sigma(G_j,K_j,n);\Z)$ pushed forward along full embeddings $(G_j,K_j)\hookrightarrow (G,K)$;
            \item For any $i,r,n \geq 0$, there exists an integer $d_{i,r,n} \geq 0$ such that for any pointed cograph $(G,K)$ with $|K| = r$, the exponent of the group $H_i(\Sigma(G,K,n);\Z)$ divides $d_{i,r,n}$;
            \item For any $i,r,n \geq 0$, and any field $k$, the function,
            \[
            m \mapsto \dim_k(H_i(\Sigma(K_m,[r],n);k))
            \]
            agrees with a polynomial for all $n \gg 0$.
    \end{enumerate}
\end{thmab}

\section*{Acknowledgments}

    The first and third authors are eternally grateful to Jonathan Nalikka, whose tremendous hard work and wonderful spirit was critical to the production of this work, and who unfortunately passed away before its completion.
    
    This work was primarily completed during the 2023-2024 season of the QED REU at York College. The authors are grateful to Rishi Nath and the other organizers of that REU for their wonderful program.

    The third author was partially supported by NSF grants DMS-2400460 and DMS-2452031.
\section{Background}

\subsection{The combinatorics of cographs}\label{sec:basics}

In this section, we briefly review the classical theory of our primary objects of interest.

\begin{definition}
    In this work, all \textbf{graphs} will be finite and simple. We use the notation $v \sim w$ to indicate that $v$ and $w$ are adjacent. We say that a subgraph $H$ of a graph $G$ is \textbf{induced} if whenever $x,y$ are vertices of $H$ which are adjacent vertices in $G$, they are necessarily adjacent in $H$ as well.

    The class of \textbf{cographs} are defined inductively as follows. The single vertex is a cograph, and if $G_1,G_2$ are cographs, then all of $G_1 \sqcup G_2, \overline{G_1},\overline{G_2}$ are also cographs, where $\overline{\bullet}$ is the edge-complement operation. It is a classically known fact (see \cite{D}, for instance) that this characterization of cographs is equivalent to the condition that $G$ not contain the path $P_4$ on 4 vertices as an induced subgraph.
\end{definition}

It is clear that all complete graphs $K_n$ are cographs, as are all complete bipartite graphs $K_{a,b}$. In fact, more generally, for any arity $r$, the complete $r$-partite graphs $K_{a_1,\ldots,a_r}$ are cographs. The class of cographs was introduced in the 1970's in a large variety of seemingly disparate contexts  (see \cites{Seinsche1974,Jung1978,Sumner1974}, for instance). From this starting point it became clear that cographs share a shockingly wide range of equivalent characterizations which make them very natural and well behaved candidates for many graph theoretic computations and algorithms as, for instance, they are known to always be perfect \cite{Seinsche1974}. Given the aforementioned characterization of  cographs as being induced $P_4$-free, one can also prove perfection of cographs using the strong perfect graph theorem \cite{chudnovsky2006strong}.

For technical reasons that we will see in later sections, an extremely useful way that one can construct and compactly denote cographs is given by what are known as cotrees. before we can define this construction, we need a bit of background on the structure theory of rooted trees.

\begin{definition}
    A \textbf{tree} is a connected, acyclic graph. A \textbf{rooted} tree is a pair $(T,v)$ of a tree along with a choice of vertex, called the \textbf{root}. Importantly, the vertex set of any rooted tree has a poset structure given by the ``distance to the root" ordering. More specifically, if $x,y$ are two vertices of a rooted tree $(T,v)$, then we say $x \leq y$ if the (unique) path in $T$ from $y$ to $v$ passes through $x$. We say that a vertex of a rooted tree is \textbf{internal} if it is not a leaf.\footnote{In the case where $T = v$ is a single vertex, the root is considered a leaf, \emph{not} an internal vertex.}

    If $(T,v)$ and $(T',v')$ are two rooted trees, then a \textbf{rooted homeomorphic embedding} from $T$ to $T'$ is an injection $\phi$ from the vertex set of $T$ to the vertex set of $T'$ which maps $v$ to $v'$, and has the property that $x \leq y$ in $T$ if and only if $\phi(x) \leq \phi(y)$ in $T'$. Note that a homemorphic embedding between rooted trees is \emph{not necessarily} a graph theoretic homomorphism, as it need not send edges to edges. If we drop the requirement that our map preserves the root, then the resulting morphism is known as a \textbf{homeomorphic embedding} from $T$ to $T'$.

    If $Q$ is a quasi-ordered set, then we say that $Q$ is \textbf{well-quasi-ordered} if it does not contain any infinite descending chains, nor does it allow for infinite anti-chains. Given a well-quasi-ordered set $Q$, with quasi-ordering $\leq_Q$, we define \textbf{$Q$-labeled} rooted trees to be triples $(T,v,\psi)$ of a tree $T$ with a chosen root $v$ and a set map $\psi$ from the vertex set of $T$ to $Q$. Given two $Q$-labeled trees, $(T,v,\psi)$ and $(T',v',\psi')$, A homeomorphic embedding $\phi:(T,v) \rightarrow (T',v')$ is said to be \textbf{label respecting} if for any vertex $x$ of $T$,
    \[
    \psi(x) \leq_Q \psi(\phi(x))
    \]

    Finally, we observe that the collection of all $Q$-labeled rooted trees is itself quasi-ordered by the relation $(T,v,\psi) \leq_{K} (T',v',\psi')$  if and only if there exists a label respecting homeomorphic embedding from $(T,v,\psi)$ to $(T',v',\psi')$\footnote{Of course, one can also define this ordering using \emph{rooted} homeomorphic embeddings. The difference between these two definitions is not entirely relevant, and does not change the statement of Kruskal's Tree Theorem \ref{Kruskal}.}
\end{definition}

Given all of the above, it becomes a natural question to ask whether the ordering $\leq_K$ is itself well-quasi-ordered given that our labels are. This is indeed the case due to the following celebrated theorem of Kruskal.

\begin{theorem}[Kruskal's Tree Theorem, \cites{kruskal1960well,Nash-Williams_1963}]\label{Kruskal}
For any well-quasi-ordered set $Q$, the collection of $Q$-labeled rooted trees is itself well-quasi-ordered by the relation $\leq_K$ defined above.
\end{theorem}

\begin{remark}
    Being that a homeomorphic embedding will generally send a tree with fewer vertices into a tree with more, the ``hard" part of proving that $\leq_K$ is a well-quasi-order is the condition that it does not permit infinite anti-chains. In other words, given any infinite collection of $Q$-labeled rooted trees, one must prove that there exists a pair in the list and a homeomorphic embedding sending one into the other which respects the labels.
\end{remark}

Kruskal's Tree theorem was originally conjectured by V\'azsonyi and proven by Kruskal in \cite{kruskal1960well}. Nash-Williams provided a different and considerably more concise proof in \cite{Nash-Williams_1963}. Kruskal's work was one of the earliest great successes in the theory of well-quasi-orders, which arguably culminated in the proof of the Graph Minor Theorem by Robertson and Seymour in \cite{robertson2004graph}. For the reader interested in the deep structural combinatorics and graph theory required to prove these results, \cite{lovasz2006graph} provides a wonderful exposition in this direction.

\begin{remark}
    A special case of Kruskal's Tree theorem that is frequently used in this work and others is the case where $Q$ is finite and given the structure of a well-quasi-order by setting $\leq_Q$ to be the equality relation. In this case, Kruskal's tree theorem can be more concretely stated as saying that given any infinite collection of rooted trees whose vertices are labeled by the elements of $Q$, there must exist a pair in the infinite list and a homeomophic embedding from one to the other, which \emph{preserves} the labels. In other words, the map must send each vertex to a vertex with the exact same label.
\end{remark}

Having stated all the necessary background in rooted trees and well-quasi-orders, we are now ready to define cotrees.

\begin{definition}
    Let $S$ denote the set
    \[
    S = \{0,1,l\}
    \]
    considered as a well-quasi-order under the equality relation. Then a \textbf{cotree} is a $S$-labeled rooted tree $(T,v,\psi)$ satisfying the following additional conditions:
    \begin{enumerate}
        \item every internal vertex is labeled with either 0 or 1, and every leaf is labeled with $l$;
        \item every internal vertex has at least two children;
        \item if an internal vertex is labeled with $0$, then all of its internal children must be labeled with $1$, whereas if an internal vertex is labeled 1, its internal children must be labeled with 0.
    \end{enumerate}
    A morphism between cotrees is a homeomorphic embedding between the underlying rooted trees which respects the labels.
\end{definition}

To see the relationship between cotrees and cographs, recall that the class of cographs has an inductive definition beginning at the single vertex, and building new cographs via complementation and disjoint union. Given a cotree $(T,v,\psi)$, we can construct a cograph inductively as follows:

\begin{itemize}
    \item if $T = v$ is the single vertex labeled with $l$, then the associated cograph is also a single vertex;
    \item if the root $v$ has children $w_1,\ldots,w_c$ with $c > 1$, and label $\psi(v) = 0$, then the associated cograph is the disjoint union of the cographs associated to the various cotrees rooted at $w_1,\ldots,w_c$;
    \item if the root $v$ has children $w_1,\ldots,w_c$ with $c > 1$, and label $\psi(v) = 1$, then the associated cograph is obtained by taking the complements of the cographs assigned to the children, taking the disjoint union of these complements, and then taking the complement of what results.
\end{itemize}

The third operation described above, when $\psi(v) = 1$, is also sometimes called the \textbf{join} of the various cographs involved. One may equivalently think about this operation as taking the disjoint union of all of the cographs associated to the children of the root, and then connecting every vertex to every other vertex not in its own cograph. One can see a collection of examples of cotrees and their corresponding cographs in Figure \ref{exampleCo}

\begin{figure}
    \caption{Some cotrees (right) and their corresponding cographs (left)}\label{exampleCo}
        \includegraphics[width=0.5\textwidth]{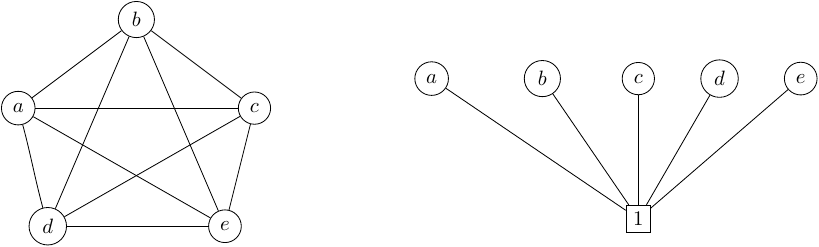}\\
        \includegraphics[width=0.5\textwidth]{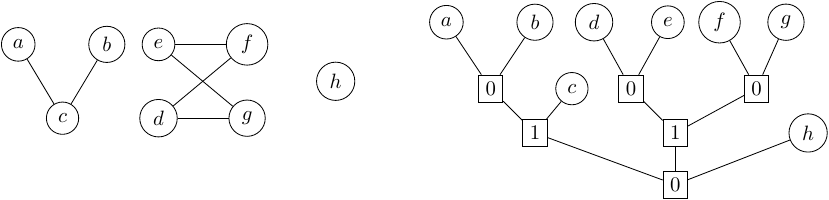}\\
        \includegraphics[width=0.5\textwidth]{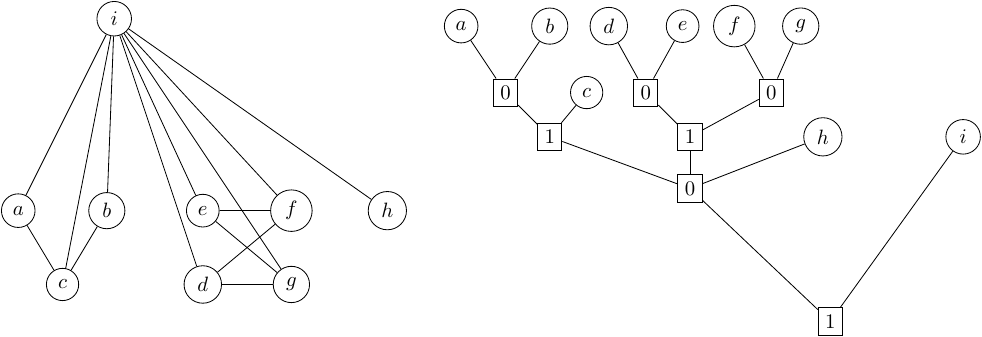}
\end{figure}

We note that by the nature of this inductive construction, the leaves of the original cotree will correspond with the vertices of the resulting cograph. Moreover, two vertices in the associated cograph will be connected by an edge if and only if the first common ancestor of corresponding leaves in the cotree is labeled with a 1. In otherwords, because morphisms between cotrees must preserve all labels, any morphism between two cotrees naturally defines a map between the corresponding cographs.

While the above operation clearly constructs a cograph from every cotree, what is considerably less obvious is that this process is reversible. In fact, something a bit stronger is true as compiled by the following theorem. For the purposes of this theorem, a \textbf{full embedding} from any graph $G$ to another graph $H$ is an injection between their vertex sets with the added property that two vertices are connected by an edge in $G$ if and only if their images are connected by an edge in $H$. In other words, $G$ fully embeds into $H$ if and only if $G$ is an induced subgraph of $H$.

\begin{theorem}\label{cotreeIsCograph}
    Let $\mathsf{Cot}$ denote the category of cotrees and morphisms between cotrees, and let $\mathsf{Co}$ denote the category of cographs and full embeddings. Then the algorithm described above defines an equivalence between $\mathsf{Cot}$ and $\mathsf{Co}$.
\end{theorem}

The fundamental and combinatorial core of Theorem \ref{cotreeIsCograph} was proven in \cite{D}, though that work did not use categorical language in its presentation. The statement as it is presented here in terms of an equivalence of categories first appeared in the precursor work \cite{KR}. We will see that the choice to frame this statement as a categorical one is in many ways the core driving point of this work, as well as parts of \cite{KR}.

Classically speaking, the primary reason why the cotree construction was considered so powerful is because of what it implies when combined with Kruskal's Tree Theorem.

\begin{corollary}\label{cographWQO}
    The collection of cographs is well-quasi-ordered by the induced subgraph relation.
\end{corollary}

Again, this classical combinatorial statement originates from \cite{D}.

\subsection{Polynomial rings on graphs}

In this section we detail a number of constructions connecting graph theory with classical algebra, specifically the algebra of polynomials.

\begin{definition}\label{def:BinIdeal}
    Let $G$ denote a graph with vertex set $V_G$ and edge set $E_G$, and fix a (commutative) Noetherian ring $k$. We will write $A_{|G|}$ to denote the \textbf{vertex polynomial ring}
    \[
    A_{|G|} = k[x_v]_{v \in V_G},
    \]
    and $A_{||G||}$ to denote the \textbf{edge polynomial ring}
    \[
    A_{||G||} = k[x_e]_{e \in E_G}.
    \]

    The \textbf{edge ideal} of $G$ is the ideal of $A_{|G|}$ generated by the (square free) monomials $x_vx_w$, where $v$ and $w$ are adjacent vertices. We can also define the \textbf{binomial edge ideal} of $G$ to be the ideal in $A_{|G|}^{\otimes 2}$ generated by binomials of the form $x_vy_w-y_vx_w$, where $v$ and $w$ are again adjacent. Finally, the \textbf{parity edge ideal} of $G$ is the ideal in $A_{|G|}^{\otimes 2}$ generated by binomials of the form $x_vx_w - y_vy_w$.
\end{definition}

The edge ideal of a graph was introduced by Villarreal in \cite{villarreal1990cohen}, and has since become a staple in a particular subgenre of commutative algebra (see \cite{morey2012edge} for a survey).

A very common thread in this work is to try and understand what the combinatorics of the graph dictates above the algebra of the edge ideal and its syzygies, and vice versa. For instance, it is easily seen that the Krull dimension of the edge ideal agrees with the size of the largest independent set of the graph (i.e. the independence number). Less obvious is the fact that the Castelnuovo-Mumford regularity of the edge ideal is bounded from above by the minimum size of a maximal matching in the graph \cite{woodroofe2014matchings}.

Binomial edge ideals were introduced in \cite{herzog2010binomial} as a common generalization of a number of objects that had been studied in algebraic statistics and commutative algebra. As with edge ideals, algebraic properties of binomial edge ideals have been shown to be closely linked with combinatorial properties of the corresponding graphs (see \cite{saeedi2016binomial} for a survey of such results). Finally, parity edge ideals were introduce quite recently in \cite{kahle2016parity}. It is noted in that work that these ideals are considerably more complicated algebraically because their natural binomial generating set is generally \emph{not} a Gr\"obner basis.

\begin{remark}
    It is extremely relevant to remark that while the edge ideals and binomial edge ideals of graphs can form a bridge between algebra and combinatorics, this bridge is not perfect. This is seen, for instance, in the fact that many relevant homological and algebraic properties of these ideals depend on the characteristic of the underlying field $k$, and therefore cannot purely be described based on the combinatorics of the graph. We therefore stress that many of the results of this paper do not depend on the characteristic of $k$ and, in fact, do not even require that $k$ be a field.
\end{remark}

In this paper, we aim to prove facts about edge ideals and binomial edge ideals associated to cographs. One pre-existing work in this direction is \cite{kahle2019binomial}, where the regularity of the binomial edge ideal associated to any cograph is upper bounded as a linear function of the number of vertices. While weaker linear bounds were known in the general graph case, Kahle and Kr\"usemann showed that these general bounds could be vastly tightened in the cograph case. 

The key argument in \cite{kahle2019binomial} is to use the fact that cographs are recusively built out of repeated unions and joins. One then argues that both of these operations reflect in a very specific way on the binomial edge ideal, and apply inductive arguments to conclude the desired bounds. In the present work, the arguments of \cite{kahle2019binomial} are ``categorified" in a sense. By doing so, we are able to recover shades of the bounds found in \cite{kahle2019binomial}, but across \emph{all} ideals in $A_{|G|}$ (or more generally $A_{|G|}^{\otimes c}$ for any constant $c$) that are preserved by the relevant cograph structures. This will be discussed in more detail in Section \ref{sec:categorical}.

Finally, one may have observed that none of the constructions discussed above relate with the edge algebra $A_{||G||}$. Indeed, we will see later that the methods of this paper fundamentally cannot work on ideals of that ring. It is mentioned here for completeness, as well as to later illustrate why the combinatorics of cographs makes working with edges rather than vertices much harder. This will be a recurring theme throughout the work.

\subsection{Graphical simplicial complexes}

In this section we discuss the foundational aspects of some of the topological applications of our machinery.

\begin{definition}
    Let $G$ be a graph with vertex set $V$. A \textbf{vertex monotone graph complex} is a functor $\Delta_\bullet$ from the category of graphs with full embeddings to the category of abstract simplicial complexes and simplicial maps, with the added requirement that the vertex set of $\Delta_G$ is (a subset of) the vertex set of $G$, for all graphs $G$. More concretely, a vertex monotone graph complex is a collection of simplicial complexes $\{\Delta_G\}$, each on (a subset of) the vertex set $V_G$ of $G$, with the added property that any full embedding $G \hookrightarrow G'$ functorally induces simplicial maps $\Delta_G \rightarrow \Delta_{G'}$.
\end{definition}

The first example of a vertex monotone graph complex is the \textbf{independence complex} $\mathcal{I}_G$. More specifically,
\[
\mathcal{I}_G^{(d)} := \{\{v_1,\ldots,v_{d+1}\} \mid v_i \not\sim v_j \forall i,j\}
\]

Independence complexes appear all throughout combinatorial topology and geometry, including quite recently (and somewhat surprisingly) in relation to varieties associated to cluster algebras \cite{lam2023cohomology}. Related to the independence complex one may also define the \textbf{clique complex} or \textbf{flag complex} of $G$, built out of collections of vertices which form a complete subgraph in $G$. Of course, the flag complex of any graph $G$ is isomorphic to the independence complex of its complement $\overline{G}$. Outside of these two there are a number of other vertex monotone complexes that one can study including,

\begin{itemize}
    \item Neighborhood complexes, built out of collections of vertices, all of whom pairwise have at least one common neighbor \cites{lovasz1978kneser,kahle2007neighborhood};
    \item Dominance complexes, built out of collections of vertices whose \emph{complements} in the full vertex set are dominating sets. That is to say, every vertex in the graph is either in the set, or adjacent to something in the set \cites{ehrenborg2006topology,matsushita2025dominance}. Note that, in this case, the vertex set of the simplicial complex may not be the entire vertex set of the graph. Indeed, isolated vertices of the graph cannot appear as vertices of the simplicial complex;
    \item $n$-cut complexes, built out of \emph{complements} of collections of vertices which contain at least one independent set of size $n$ \cite{bayer2024topology}. Note that just as with the last case, the vertex set of the complex is not necessarily the entire vertex set of the graph, in the case where a vertex appears in all $n$-element independent sets;
    \item If $\mathcal{P}$ is any \textbf{hereditary} graph property (i.e. one that descends to induced subgraphs), then one may form a complex associated to an arbitrary graph $G$, by having the simplicies correspond to induced subgraphs of $G$ with the property $\mathcal{P}$. Note that the independence and clique complexes are examples of such graph complexes, where the property $\mathcal{P}$ is ``does not have edges," or ``is a complete graph," respectively. On the other hand, the neighborhood complex is not associated to a hereditary property.
    \item Hom-complexes\footnote{These complexes are actually polyhedral rather than simplicial. Most of the results of this paper will still apply for identical reasons.} built out of multi-homomorphisms between two fixed graphs;
\end{itemize}

In this work, we will appeal to the independence and clique complexes as natural -- and it will turn out in our case, simple -- examples of vertex monotone simplicial complexes that one can associate to cographs. To see why these complexes are simple in the cograph case, observe that for any two graphs $G,G'$,
\begin{align*}
\mathcal{I}(G \sqcup G') \cong \mathcal{I}(G) \star \mathcal{I}(G)\\
\mathcal{I}(G \star G') \cong \mathcal{I}(G) \sqcup \mathcal{I}(G),
\end{align*}
where $\star$ is the simplicial join operation. It follows that, through usage of things like K\"unneth and standard long exact sequences from topology, the homology groups of the independence complex of any cograph is naturally built out of the homology groups of its building blocks through the cotree construction (Section \ref{sec:basics}). In other words, \textbf{the homology groups of the independence complexes of a general cograph can always be described purely in terms of homology groups of smaller constituent cographs.} One can think of this as a kind of topological finite generation, which we will make more precise in the next section.

Just as the work of \cite{kahle2019binomial} provides a combinatorial framework for studying the bionomial edge ideals of cographs, which we categorify here for the sake of expansion and generalization, the above simple argument illustrates one very special case of what we will accomplish for more general graph complexes. While independence complexes and clique complexes are made simple by their relationships to the join and union operations, it is certainly not the case that \emph{every} vertex monotone complex will behave in a similar fashion. Despite this level of generality, our methods allow for one to prove similar topological finite generation results across all cographs.

Prior to this work, there were a number of papers that viewed graph complexes built out of the \emph{edges} of the graph rather than the vertices using the language of graph category theory (See \cites{miyata,caputi2024finite,caputi2024weak}, for instance). As previously stated, it will turn out to be much more natural to work with the vertex set of a cograph rather than its edges, and so the results of those works can be viewed as genuinely disjoint from this work.

\subsection{The representation theory of combinatorial categories} \label{sec:categorical}

Having completed our background discussion on the more ``classical" combinatorial and algebraic portions of this work, we now take our time to discuss the key technical underpinnings of our primary results.

\begin{definition}\label{def:Catrep}
Let $\mathcal{C}$ be a category, and let $k$ denote any (commutative, unital) ring. Then a \textbf{representation of $\mathcal{C}$} (over $k$) is a functor $V:\mathcal{C} \rightarrow k-\text{Mod}$, where $k$-Mod is the category of $k$-modules.

For much of what we do, the ring $k$ will not play a major role. By consequence, we will frequently just refer to representations of $\mathcal{C}$ without specifying the ring that they are over.

If $A$ is an object of $\mathcal{C}$, then we will write $V_A$ for the valuation of $V$ at $A$. We will also frequently refer to the \textbf{induced maps}, or \textbf{transition maps}, as the images of morphisms of $\mathcal{C}$ under $V$.

The category of $\mathcal{C}$-modules with natural transformations as morphisms is an abelian category, with kernels, cokernels, and direct sums defined point-wise on objects. For this reason we will use module-theoretic language such as \textbf{submodule} and \textbf{quotient module} freely.
\end{definition}

For those more familiar with the more typical representation theory of groups, the above definition is a direct generalization. Indeed, viewing a group $G$ as itself being a category with a single object, a representation of $G$ is precisely a functor from this category to the category of $k$-modules.

The representation theory of categories has existed for many decades to this point (see \cite{webb2007and} for a survey), though it has seen an explosion of interest in the last ten years or so due to its connection with representation stability theory \cites{church2013representation,churchFIMod}. This ``new wave" of representation theoretic literature primarily focuses on the representation theory of what are usually referred to as \textbf{combinatorial categories} \cite{sam2017grobner}. Though not really a precise concept, a combinatorial category is one whose objects and morphisms encode concrete combinatorial data. For instance:

\begin{enumerate}
    \item The category FI of finite sets and injective maps. The representation theory has been studied in \cites{churchFIMod,church2014fi,wilson2018introduction}.

    \item Expanding the prior example, categories of finite sets with ``decorated" injections. For instance, the category FI$_d$ whose morphisms are injections with a $d$-coloring on their complements \cites{ramos2017generalized,alpert2020generalized}, and FIM$^+$, whose morphisms are a pair of an injection with a choice of perfect matching on the compliment \cite{miller2019higher}.

    \item The opposite category FS$^{op}$ of finite sets with surjective maps. The representation theory has been studied in \cites{proudfoot2017configuration,tosteson2021stability}.

    \item The opposite category of (rooted) trees with edge contractions. More generally, the opposite category of graphs with bounded genus and edge contractions, $\mathcal{G}_g^{op}$. The representation theory has been studied in \cites{barter2015noetherianity,proudfoot2019functorial,proudfoot2022contraction,miyata}.
\end{enumerate}

In this work we will be expanding on things proven in \cite{KR}. More specifically, we will consider representations of the category $\mathsf{Co}$.

Just as with group representation theory, one must determine a ``nice" collection of representations to focus our attention on.

\begin{definition}
    Let $\mathcal{C}$ be a category, and $V$ a representation of $\mathcal{C}$. We say that $V$ is \textbf{finitely generated} if  $V_A$ is finitely generated as a $k$-module for every $A$ and there exists a finite collection of objects $A_1,\ldots,A_n$ in $\mathcal{C}$, such that for any object $A$ of $\mathcal{C}$, the module $V_A$ is generated by the images of the modules $V_{A_i}$ under the various induced maps  $A_i \rightarrow A$ 
\end{definition}

\begin{example}
    The following example will become important to us moving forward. For any cograph $G$, one may associate a $k$-module $kV_G$ by linearizing the vertex set. This naturally extends to a $\Co$-module, $kV_\bullet$, which is finitely generated by the single vertex cograph. More generally, for any fixed integer $i$, the wedge power $\bigwedge^ikV_\bullet$ is also finitely generated by the cographs with exactly $i$ vertices. The same arguments apply to the linearizations of (wedge powers of) the edge set.
\end{example}

\begin{theorem}[Noetherianity for Cograph representations, \cite{KR}]\label{coNoeth}
Assume that $k$ is a Noetherian ring, and let $V$ be a finitely generated $\mathsf{Co}$-module over $k$. Then all submodules of $V$ must also be finitely generated.
\end{theorem}

\begin{example}
    As an elementary, but important, application of this theorem, consider the $\mathsf{Co}$-module $T$, defined on objects by
    \[
    T_G = \mathbb{Z} \text{ for all cographs $G$},
    \]
    and defined on full embeddings by assigning the identity map. If one takes any infinite collection of cographs $\mathcal{F}= \{G_\alpha\}_{\alpha in I}$, then there is a submodule $V^{\mathcal{F}} \subseteq V$ obtained by setting
    \[
    T^\mathcal{F}_G = \begin{cases} \mathbf{Z} &\text{ if $G_\alpha$ is induced in $G$ for some $\alpha$}\\
    0 &\text{else.}\end{cases}
    \]

    By our Noetherianity theorem, it follows that $T^{\mathcal{F}}$ is finitely generated, as $T$ is clearly finitely generated by the single vertex graph. In particular, there must be a \emph{finite} subcollection $\mathcal{F}' \subseteq \mathcal{F}$ such that for any $G_\alpha \in \mathcal{F}$, $G_\alpha$ contains a member of $\mathcal{F}'$ as an induced subgraph. Thus, $\mathcal{F}$ cannot be an anti-chain with respect to the induced subgraph order.

    This example illustrates why prior works have described these types of categorical representation Noetherianity statements as \textbf{categorifications} of more classical well-quasi-order theorems.
\end{example}

While the Noetherianity theorem is already powerful enough to prove a number consequences in combinatorial topology, for our purposes we will need to strengthen it in the following way.

\begin{definition}\label{def:cographPolyRing}
    Let $A_{|\bullet|}$ denote the functor from $\mathsf{Co}$ to the category of $k$-algebras which is defined on objects by sending a cograph to its vertex polynomial ring. Because the maps of $\mathsf{Co}$ are injective on vertices by definition, the induced maps on $A_{|\bullet|}$ are naturally defined.

    An \textbf{$A_{|\bullet|}$-module} $V$ is a $\mathsf{Co}$-module such that for every cograph $G$, $V_G$ carries the structure of a $A_{|G|}$-module, and for every full embedding $\phi:G \hookrightarrow G'$ and every vertex $v$ of $G$, the diagram,
    \[
    \begin{CD}
        V_G @>>\phi> V_{G'}\\
        @VV\cdot x_vV   @VV\cdot x_{\phi(v)}V\\
        V_G @>> \phi>  V_{G'}
    \end{CD}
    \]
    commutes. \footnote{obviously the words ``graded" or ``multi-graded" can be added to this definition without much effort. In these cases we require that the induced maps preserve whatever grading is imposed.}

    We say that an $A_{|\bullet|}$-module $V$ is \textbf{finitely generated} if $V_{G}$ is a finitely generated $A_{|G|}$-module for every cograph $G$, and there exists a finite list of cographs $G_1,\ldots,G_n$ such that for any cograph $G$, $V_G$ is generated (as an $A_{|G|}$-module) by the images of $V_{G_i}$ under the transition maps.
\end{definition}

Our primary technical theorem expands the above Noetherianity statement \ref{coNoeth} by showing that submodules of finitely generated $A_{|\bullet|}$-modules are always themselves finitely generated. Note that any finitely generated $\mathsf{Co}$-module can be upgraded to a finitely generated $A_{|\bullet|}$-module by making the variables act trivially. It follows that this new Noetherianity statement directly implies Theorem \ref{coNoeth}

\begin{remark}
    One of the most attractive features of $\FI$-modules over fields of characteristic 0 is the way in which the irreducible representations of the grades pieces grow \cite{churchFIMod}. In a very recent paper \cite{scarabotti2024representation}, Scarabotti provided a description of the irreducible representations associated to the automorphism groups of (finite) rooted trees. This description was in terms of a kind of ``tree partition," and extended the very classical bijection between irreducible representations of the symmetric groups and integer partitions. In progress work of the third author and Joshua Birns will show that modules (in characteristic 0) over the category of rooted trees display similar multiplicty-stability behaviors as $\FI$-modules.

    In Section \ref{sec:basics}, we discussed the equivalence between the category of cographs, and the category of cotrees. This equivalence allows one to consider cographs as a kind of labelled (rooted) tree, and its morphisms as a kind of tree automorphism that preserves the labels. Now while that might seem to imply that cograph autmorphism groups are \emph{proper} subgroups of tree automorphism groups, this is actually not the case. This is because cotree labels are entirely determined by whether or not you are a leaf, and what distance you are from the root. Both of these features are always preserved by any tree autmorphism.

    All of the above is to say, being a finitely generated module over $\Co$ (in characteristic 0) implies a kind of multiplicity stability result that directly extends the original multiplicity result of $\FI$-modules. One should keep this in mind when observing the various applications of Section \ref{sec:app}.
\end{remark}

\section{The proof of the main technical theorem ~\ref{mainTech}} 

\subsection{FI-concrete categories and representation stability}

In order to prove our primary technical theorem, we will need to introduce new machinery developed by Laudone and Snowden \cite{laudone}. Equivalent machinery under a different name and notation was considered at around the same time by Miyata and the third author \cite{miyata}, though the approach of Laudone and Snowden is considerably more compact and easily presented, so we will use it here.

\begin{definition}
    We say that a pair $(\mathcal{C},|\bullet|)$ of a category $\mathcal{C}$ with a functor $|\bullet|:\mathcal{C} \rightarrow \FI$ is \textbf{$\FI$-concrete}, if the functor $|\bullet|$ is faithful and \textbf{conservative}, in that the only maps of $\mathcal{C}$ that it sends to isomorphisms in $\FI$ are isomorphisms of $\mathcal{C}$. A functor between $\FI$-concrete categories $\Phi$ is itself said to be \textbf{concrete} if there are natural isomorphisms
    \[
    |A| \cong |\Phi(A)|
    \]  

    In all cases that follow, the functor $|\bullet|$ should be obvious from context. As such, we will frequently just refer to $\mathcal{C}$ as being $\FI$-concrete without explicitly mentioning $|\bullet|$.
\end{definition}

One may think of $\FI$-concrete categories as combinatorial categories whose objects can be naturally realized as finite sets, in such a way that one does not lose too much of the underlying combinatorics by doing so. For instance, the category $\mathsf{Co}$ is $\FI$-concrete, where $|G|$ is defined to be the vertex set of $G$. Obviously $\FI$ itself is $\FI$-concrete by setting $|\bullet|$ to be the identity functor.

\begin{definition}
    Let $\mathcal{C}$ be an $\FI$-concrete category and fix a field $k$. Then the \textbf{polynomial ring over $\mathcal{C}$} is the functor $A_{|\bullet|}:\mathcal{C} \rightarrow k-\text{Alg}$ defined on objects by
    \[
    A_{|B|} = k[x_{b} \mid b \in |B|]
    \]

    A \textbf{module} over $A_{|\bullet|}$ is defined analogously to Definition \ref{def:cographPolyRing}. For any object $B$ of $\mathcal{C}$, we define the \textbf{free $A_{|\bullet|}$-module relative to B} by the assignment
    \[
    F_B(C) = A_{|C|}^{\oplus |\Hom_\mathcal{C}(B,C)|}.
    \]
    More generally, an $A_{\bullet}$-module is said to be \textbf{free} if it is isomorphic to a direct sum of modules that are free relative to some collection of objects. In this case, we refer to these objects as the \textbf{generators} of the free module.

    We say that an $A_{|\bullet|}$-module $M$ is \textbf{finitely generated} if it is surjected onto by a free module with only finitely many summands. The generators of this free module then become the generators of the module $M$.
\end{definition}

In their seminal work \cite{sam2017grobner}, Sam and Snowden introduced a kind of Gr\"obner theory for combinatorial categories. This theory provided an incredibly powerful and truly combinatorial way to prove that the representation theory of a given combinatorial category was Noetherian. In moving from proving that the algebra of functors $\mathcal{C} \rightarrow k-Mod$ has a Noetherian property, to proving that modules over $A_{|\bullet|}$ do, extra care must be paid. The underlying ideas that motivated \cite{laudone}, and earlier works such as \cite{miyata} and \cite{nagel2019fi}, was that for a given $A_{|\bullet|}$-module $V_{\bullet}$, one must construct Gr\"obner bases -- in the classical polynomial sense -- for every module $V_B$ simultaneously, while maintaining compatibility with a Gr\"obner basis -- in the Sam-Snowden sense -- for the underlying categorical representation. One accomplishes this by considering the objects of the category $\mathcal{C}$ that have been weighted in an appropriate sense.

\begin{definition}
    Let $\mathcal{C}$ denote an $\FI$-concrete category, with objects $B, C,$ and $D$ and a morphism $f:B \rightarrow C$. A \textbf{weight} on $C$ is a set function $\gamma:|C| \rightarrow \mathbb{N}$, whereas a \textbf{weight on $C$ relative to $f$} is a pair $(\gamma,f)$, where $\gamma$ is a weight on $C$. Given weights $\gamma,\delta$ on $C$ and $D$, respectfully, a \textbf{weighted morphism} from $C$ to $D$ is a map $\phi:B \rightarrow C$ such that,
    \[
    \beta(b) \leq \gamma(|\phi|(b)),
    \]
    for all $b \in |B|$. Finally, if $(\gamma,f)$ and $(\delta,g)$ are weights on $C$ and $D$, respectively, relative to $f:B \rightarrow C$ and $g:B \rightarrow D$, respectively, a \textbf{weighted morphism} between them is a weighted morphism $\phi:C \rightarrow D$ such that $g = \phi \circ f$.
    
    We write $\mathcal{M}(\mathcal{C})$ for the set of (isomorphism classes of) weighted objects on $\mathcal{C}$. This set is equipped with the natural structure of a poset, by setting $[B,\beta] \leq [C,\gamma]$ if and only if there is a weighted morphism from $(B,\beta)$ to $(C,\gamma)$\footnote{As noted by \cite{laudone}, anti-symmetry of this partial order explicitly uses the fact that $|\bullet|$ is conservative.} For any fixed object $B$ of $\mathcal{C}$, we write $\mathcal{M}(\mathcal{C},B)$ for the set of (isomorphism classes) of weighted objects relative to morphisms with domain $B$. As with $\mathcal{M}(\mathcal{C})$, all of these sets carry a natural poset structure.
\end{definition}

The important thing to note about the above definitions is that a given weighted object $(B,\beta)$ can be equivalently thought of as a monomial in $A_{|B|}$. In this context, the poset of weighted objects $\mathcal{M}(\mathcal{C})$ is a divisibility poset of monomials across many different -- albeit connected through the induced maps coming from $\mathcal{C}$ -- polynomial rings. On the other hand, relatively weighted objects can be thought of as monomials in the free module $F_B$

\begin{definition}
    We say that an $\FI$-concrete category $\mathcal{C}$ is a \textbf{G-category}, if:
    \begin{enumerate}
        \item it is possible to assign to each set $|B|$ a total order, in such a way that the induced maps $|\phi|:|B| \rightarrow |C|$ are all order-preserving and,
        \item the poset $\mathcal{M}(\mathcal{C})$ of weighted objects in $\mathcal{C}$ is well-quasi-ordered.
    \end{enumerate}

    We say that an $\FI$-concrete category $\mathcal{C}$ is a \textbf{strongly G-category} if it is a G-category in the above sense with the additional conditions that:
    \begin{enumerate}
        \item $\mathcal{C}$ is a \textbf{Gr\"obner category} in the sense of \cite{sam2017grobner}. That is to say, for every object $B$, there exist total orders on the set of isomorphism classes of morphisms with domain $B$ which are compatible with post-composition in the obvious way.
        \item For every object $B$ of the category, the poset $\mathcal{M}(\mathcal{C},B)$ of relatively weighted objects in $\mathcal{C}$ is well-quasi-ordered.
    \end{enumerate}

    We call an $\FI$-concrete category $\mathcal{D}$ a \textbf{QG-category} (resp. \textbf{strongly QG-category}) if there exists a G-category (resp. strongly G-category) $\mathcal{C}$ and an essentially surjective concrete functor $\mathcal{C} \rightarrow \mathcal{D}$ (resp. essentially surjective full concrete functor). In this case we will often refer to $\mathcal{C}$ as a \textbf{G-cover} for $\mathcal{D}$.
\end{definition}

\begin{remark}
    In the work \cite{laudone}, Laudone and Snowdon only consider G and QG categories, though they note that their notion can be combined with the Gr\"obner concepts from \cite{sam2017grobner} to produce stronger results. The definition given above for the "strong" versions of these ideas is precisely what they mean by this. We will see below what the practical difference is between strongly G and G categories.
\end{remark}

In their seminal work \cite{sam2017grobner}, Sam and Snowden introduce what are known as Gr\"obner categories. While the definition (and nomenclature) are very similar to that of G-categories, they are not the same thing. Indeed, the authors of \cite{laudone} provide an example of an $\FI$-concrete Gr\"obner category which is not a G-category.

It is perhaps also important to note that whether or not an $\FI$-concrete category is a G-category (or a QG-category) depends on both the underlying category as well as the functor $|\bullet|$. For instance, the primary technical theorem of this work will show that the category $\mathsf{Co}$ with $|G|$ defined to be the set of vertices of $G$, is an $\FI$-concrete QG-category (See Theorem \ref{mainTech}). On the other hand, if one selects $|\bullet| = ||\bullet||$ to denote the \emph{edge set} of the cograph $G$, then the resulting category is $\FI$-concrete, but not QG. To see why this is, we need to first state the main theorem from \cite{laudone}.

\begin{theorem}[Laudone and Snowden, Theorem 1.4]\label{thm:QGNoeth}
If an $\FI$-concrete category $\mathcal{C}$ is either a G-category or a QG-category, then all submodules (i.e ideals) of $A_{|\bullet|}$ are finitely generated. More generally, if $\mathcal{C}$ is either a strongly G-category or a strongly QG-category, then all submodules of any finitely generated $A_{|\bullet|}$-module are finitely generated. 
\end{theorem}

\begin{proof}
    To start, assume that our category is a G-category. The first half of this theorem is precisely Theorem 1.4 of \cite{laudone}. The crux of that proof is that the first condition of being a G-category precisely allow one to define monomial orders on all $A_{|B|}$ simultaneously, which allows one to reduce the problem to monomial ideals of $A_{|\bullet|}$. The second condition then implies that any monomial ideal in $A_{|\bullet|}$ is finitely generated.

    In the strongly G case, the above argument can be precisely replicated for any of the free modules $F_B$. Indeed, one can define monomial orders by first using the provided total order on $\mathcal{M}(B)$ given by the first condition of being a strongly G category, and then using the total order given by the first condition of being a G-category. The second condition of being a strongly G category then allows one to show that monomial submodules of $F_B$ are always finitely generated. Usual algebra tricks can then be used to conclude that all submodules of finitely generated modules are finitely generated, as we have shown this to be the case for all free modules.

    To show the necessary results for (strongly) QG-categories, one argues as in \cite{laudone} or \cite{sam2017grobner}, that the existence of a G (or Gr\"obner) cover allows one to reduce the question of finite generation to the restriction of the module to the action of that cover.
\end{proof}

\begin{example}
    As a first example of this theorem, consider $\FI$ itself as an $\FI$-concrete category with $|\bullet|$ being the identity functor. It is not hard to show that $\FI$ is a (strongly) QG-category, with G-cover being given by the category $\OI$ of totally ordered finite sets with order preserving injections. In this case, \ref{thm:QGNoeth} implies the following: if $I_1 \leq I_2 \leq \ldots$ is a symmetric chain of ideals, in the sense that each $I_n \leq k[x_1,\ldots,x_n]$ is preserved by the symmetric group $S_n$, then there exists a non-negative integer $m$ and a finite collection of polynomials $f_1,\ldots f_r \in k[x_1, \ldots, x_m]$ such that $I_n$ is generated by the $S_n$--orbits of $f_1, \ldots f_r$ for all $n \geq m$. A version of this theorem has been known as far back as Cohen \cite{cohen1967laws}, though its formulation in terms of the category $\FI$ is as recent as \cite{nagel2019fi}.\\

    As a non-example, consider the category $\mathsf{Co}$, with functor $|\bullet| = ||\bullet||$ given by the \emph{edge set}. To show that this $\FI$-concrete category is not QG, we will argue that the polynomial ring over this category is not Noetherian in the sense of Theorem \ref{thm:QGNoeth}.

    For each $n$, the polynomial ring evaluated on complete graph $K_n$ is given by $A_{||K_n||} = k[x_{i,j}]$. Within these rings one has the chain of ideals $I_2 \leq I_3 \leq \ldots$, where $I_n$ is generated inductively by $I_{n-1}$ and the monomial $x_{1,2}x_{2,3}\cdots x_{n-1,n}x_{n,1}$. One can see that this chain of ideals cannot be finitely generated, as for each $n$, $I_n$ is generated by monomials and their $S_n$ translates. A monomial is contained in a monomial ideal if and only if it is divisible by one of the generators, and in this case the newly added generator to $I_n$ is never divisible by any of the translates of the generators for $I_{n-1}$.
\end{example}

\subsection{Proving that $\mathsf{Co}$ is a strongly QG-category}

With the background of the prior section, proving Theorem \ref{thm:MainTech} now reduces to showing that $\mathsf{Co}$ is a strongly QG-category. Theorem \ref{cotreeIsCograph} tells us that $\mathsf{Co}$ is equivalent to the category of cotrees $\mathsf{Cot}$. In fact, this equivalence can be extended to an equivalence of FI-concrete categories by Defining $|\bullet|$ on $\mathsf{Cot}$ to be the functor which assigns to every cotree its set of leaves. It therefore suffices to prove that the FI-concrete category $\mathsf{Cot}$ is a strongly QG-category. To accomplish this, we must design a G-cover for $\mathsf{Cot}$.

\begin{definition}
    A \textbf{planar rooted tree} is a rooted tree, such that at every internal vertex the children of that vertex are totally ordered. The vertices of any planar rooted tree are totally ordered by beginning at the root and applying a depth-first labeling, using the provided total ordering of children at each vertex.
    
    A \textbf{homeomorphic embedding between planar rooted trees} is a homoeomorphic embedding of the underlying rooted trees which is also monotone with respect to the aforementioned total ordering on vertices.

    We will write $\PCot$ for the category of planar rooted \emph{cotrees} with homeomorphic embeddings.
\end{definition}

Our primary objective in this section will be to show that $\PCot$ is a strongly G-category. Note that the forgetful functor $\PCot \rightarrow \Cot$ clearly realizes $\Cot$ as a strongly QG-category in this case.

Our choice of G-cover is largely inspired by the original work of Barter \cite{barter2015noetherianity}, which was later used in similar contexts at various times by Caputi, Collari, Knudsen, Miyata, Proudfoot, and the third author, \cite{proudfoot2019functorial,proudfoot2022contraction,miyata,KR,caputi2024weak}. In the work \cite{laudone}, Laudone and Snowden consider categories of what they call Boron trees, and use a slightly different G-cover. In that work, the covering category still embeds the tree in the plane, but makes sure to do so in a way in which the leaves appear on the unit circle.

We begin our analysis of $\PCot$ by proving that its collection of relatively weighted objects forms a well-quasi-order.

\begin{proposition} 
For any planar rooted tree $B$, The collection of relatively weighted objects $\cM(\PCot,B)$ is a well-quasi-order.
\end{proposition}

\begin{proof}
As is frequently the case with these types of proofs, we will approach this proposition using the standard Nash-Williams technique of minimal bad sequences \cite{Nash-Williams_1963}.

Note that a weighted object in $\PCot$ can be thought of as a planar rooted cotree, whose leaves have all been assigned natural number weights. On the other hand, if $f:B \rightarrow C$ is a morphism of planar rooted trees with domain $B$, then this data is also easily encoded as a labeling on the vertices of $C$ by using the (finite and unvarying) label set $V_B \cup \{\emptyset\}$. Indeed, one can use the vertices of $B$ to indicate the image of $\phi$, and the empty set marker to indicate vertices not in the image. The weighted arrows in $\PCot$ then become homeomorphic embeddings between the underlying rooted trees which, 
\begin{enumerate}
    \item preserve the label of all internal vertices and map leaves to leaves,
    \item preserve the total ordering on vertices given by the depth-first labeling, and
    \item map a leaf with a given weight to one with a weight at least as large.
\end{enumerate}

With the above perspective in mind, what we are trying to prove falls out as a consequence of the planar labeled version of Kruskal's Tree Theorem. The planar, but unlabeled, Kruskal's Theorem was originally proven in \cite{barter2015noetherianity}, whereas the labeled version was originally recorded in \cite{proudfoot2022contraction}. Both proofs are ultimately inspired by Nash-Williams' seminal work \cite{Nash-Williams_1963}. We briefly write an outline here for completion.

Suppose that $\cM(\PCot)$ is \emph{not} a well-quasi-order. Then there exists an antichain of (planar rooted) cotrees

\begin{align*}
T_1, T_2, T_3,\ldots 
\end{align*}

which is \emph{minimal} in the sense that, for any $n \geq 1$, $|T_n|$ is smallest among the $n$-th terms of any antichain whose first $n-1$ terms are $T_1,\ldots,T_{n-1}$.

Write $\mathcal{B}_i$ to denote the (finite) sequence of trees
\[
\mathcal{B}_i = (B_{ij})_j
\]
where $B_{ij}$ is the $j$-th branch from the root (according to the planar structure) of the $i$-th tree in the antichain, and let 
\[
\mathcal{B} = \{B_{i,j}\}_{i,j}.
\]

Since we supposed that the above antichain is minimal, $\mathcal{B}$ cannot contain an infinite anti chain, as $B_{ij} < T_i$ for each $i,j$. Thus, $\mathcal{B}$ is a well-quasi order. In particular, by Higman's lemma, the set of finite ordered sequences valued in $\mathcal{B}$ is well-quasi-ordered under the relation of embedding of words. Now, consider the sequence of sequences:
\[
\mathcal{B}_1,\mathcal{B}_2,\mathcal{B}_3,\ldots
\]

In any well-quasi-order, an infinite sequence must have some infinite chain as a subsequence. So, by Higman's lemma, there is some infinite subchain of this sequence:
\[
\mathcal{B}_{i_1} \leq \mathcal{B}_{i_2}\leq \mathcal{B}_{i_3}\leq \dots
\]
Suppose that, for each pair of branch sets $\mathcal{B}_{i_m}$ and $\mathcal{B}_{i_n}$ in this chain, the roots of the trees $T_{i_m}$ and $T_{i_n}$ have different labels. This is impossible as there are only 2 choices for the labeling of the root. Thus, there exists some $m<n$ where $\mathcal{B}_{i_m} \leq \mathcal{B}_{i_n}$ and the labeling on the roots is the same of the corresponding trees. In particular, we can embed each branch of $T_{i_m}$ into a branch of $T_{i_n}$ while preserving our planar embedding, and so we can reconstruct a map of cotrees $T_{i_m}\to T_{i_n}$ by taking the respective embeddings on the branches and sending the root to the root. Thus $T_{i_m} \leq T_{i_n}$, a contradiction.
\end{proof}

Hence, $\PCot$ is a $G$-cover of $\Co$, and by Theorems 2.5 and 2.9 of \cite{laudone} we arrive at our main theorem.

\begin{theorem}
\label{mainTech}
The polynomial ring functor $A_{|\bullet|} : \mathsf{Co}\to \mathsf{Alg}_k$ is Noetherian. 
\end{theorem}

\section{Applications}\label{sec:app}

\subsection{Commutative algebra of cographs}

We begin our discussion of some applications of Theorem \ref{mainTech} to concerns in commutative algebra. To start, we first state a slight generalization of our main theorem, which will allow us to work with both edge-ideals and binomial edge ideals.

\begin{theorem}\label{generalizedMain}
Let $c$ be a non-negative integer. Then the $c$-th tensor power of the polynomial ring functor
\[
A^{\otimes c}_{|\bullet|}
\]
is Noetherian
\end{theorem}

\begin{proof}
    The idea here is essentially present in \cite[Remark 2.2]{laudone}. Consider the $\FI$-concrete category on $\Co$ with functor $|\cdot|_c: \Co \rightarrow \FI$
    \[
    |G|_c = |G| \times [c],
    \]
    where $|G|$ is the set of vertices of $G$. One can check that the proof of \ref{mainTech} provided in the previous section can be easily adapted to prove that this new $\FI$-concrete category is QG. Indeed, our weights now assign a (fixed and finite) list of natural numbers to each vertex, rather than a single number. This extra bit of information, being finite and fixed, can be accounted for using the same Nash-Williams argument.
\end{proof}

For us, the case $c = 2$ will be the most relevant, as the binomial edge ideals are submodules of
\[
A^{\otimes 2}_{|G|} = k[x_v,y_w \mid v,w \in |G|],
\]
however we stress that all of what follows will hold for any integer $c \geq 0$.

We will be primarily focused on the behaviors of syzygy modules of $A^{\otimes c}_{|\bullet|}$-modules. Because the theory of syzygies over polynomial rings is so enhanced by the existence of minimal resolutions, we will assume in this section that \textbf{all of the $A^c_{\bullet}$-modules being considered are over a field $k = \mathbb{F}$.}

Recall that, if $\F[x_1,\ldots,x_n]$ is a polynomial ring over a field $\F$, and $M$ is a module over this polynomial ring, then the \textbf{syzygy modules} of $M$ can be realized as the terms in a minimal free resolution of $M$. By Nakayama's lemma, the generators of these terms can be viewed as the basis vectors of the Tor vector spaces $\Tor_i(M,\F[x_1,\ldots,x_m]/\mathfrak{m})$, where $\mathfrak{m}$ is the ideal generated by the variables $x_i$. This latter description is ideal in that it allows one to count syzygies without needing to explicitly construct a minimal free resolution.

Importantly, if one considers the polynomial ring as graded (by degree) or multi-graded (by setting $x_i$ to have grade $e_i$), then all of the above can be adapted to respect this extra structure. In particular, the vector space $\Tor_i(M,\F[x_1,\ldots,x_m]/\mathfrak{m})$ can be viewed as a graded, or multi-graded, vector space.

\begin{definition}
Let $G$ be a cograph, and $M$ a graded  module over $A_{|G|}$. Then for any $i \geq 0$ and integer $a$, we define the Betti number
\[
\beta^{i,a}_G(M) = \dim_k \Tor_i(M,k=\F[x_1,\ldots,x_m]/\mathfrak{m})_a.
\]
If instead $M$ is multi-graded, then for any $i \geq 0$ and non-negative integral vector $\mathbf{a}$, with coordinates indexed by $|G|$, we define the multi-graded Betti number
\[
\beta^{i,\mathbf{a}}_G(M) = \dim_k \Tor_i(M,\F[x_1,\ldots,x_m]/\mathfrak{m})_{\mathbf{a}}
\]
\end{definition}

\begin{remark}
    Of course, all of the above can be adapted to the tensor powers $A_{|G|}^{\otimes c}$ as well. In this setting there are a number of natural gradings, including by total degree, multi-degree indexed by the vertices of $G$, and multi-degree indexed by the variables of $A_{|G|}^{\otimes c}$. In all cases, the morphisms induced by full embeddings $G \rightarrow G'$ will respect these gradings. In particular, one may consider the categories of (multi-)graded modules over the (multi-)graded polynomial ring over $\Co$. These categories will have a Noetherian property as a consequence of Theorem \ref{generalizedMain}. 
\end{remark}

\begin{proposition}
    For any fixed integers $i,c \geq 0$, and any $A_{|G|}^{\otimes c}$-module $M$, the assignment
    \[
    G \mapsto \Tor_i(M_G,A_{|G|}^{\otimes c}/\mathfrak{m})
    \]
    defines a finitely generated $\Co$-module over $\F$. This continues to hold if one considers any of the natural (multi)-gradings on $A_{|G|}$, and works with correspondingly graded $\Co$-modules over $\F$.
\end{proposition}

\begin{proof}
    This is an immediate consequence of Theorem \ref{generalizedMain}, as well as the fact that the functor $- \otimes A_{\bullet}^{\otimes c}/\mathfrak{m}$ from $A_{\bullet}^{\otimes c}$-modules over $\F$ to $\Co$-modules over $\F$ clearly preserves finite generation. Because all morphisms induced from $\Co$ preserve the natural (multi-)gradings on $A_{\bullet}^{\otimes c}$, the second part of the proposition is also immediate.
\end{proof}

\begin{corollary} \label{ComAlgCor}
    Let $M$ denote an $A_{|\bullet|}$-module, equipped with the usual polynomial ring grading. Then for any $i \geq 0$, there exists an $N \gg 0$, such that for all cographs $G$, $\beta_G^{i,j}(M) = 0$ for all  $j \geq N$. 
    
    If we assume instead that $M$ is equipped with the natural $|G|$ multi-grading of $A_{|\bullet|}$, then for all $i$ there exists a finite list of multi-degrees $\mathbf{a}_1,\ldots,\mathbf{a}_n$ such that for any cograph $G$ and any multi-index $\mathbf{a}$ not on this list, $\beta^{i,\mathbf{a}}_G(M) = 0$
\end{corollary}

Of course, the above corollary can also be stated in terms of the various (multi)graded modules over $A_{\bullet}^{\otimes c}$. The way that one should think about Corollary \ref{ComAlgCor} is that in any module over the cograph polynomial ring, in any fixed index the syzygies in that index are ultimately controlled by the syzygies coming from a finite list of cographs. 

\begin{remark}
    In \cite{fieldsteel2025minimal}, Fieldsteel and Nagel consider the case of the polynomial ring over $\FI$. More specifically, they study instances wherein one can build a minimal free resolution of an $A_{|\bullet|}$-module, which is coherent in the sense that it restricts to a minimal free resolution for each $n \geq 0$. It is perhaps an interesting avenue of future study to determine whether their methods can produce similar results in our setting of the polynomial ring over $\Co$.
\end{remark}

One of the main results of \cite{kahle2019binomial} is a strict bound on the regularity of the binomial edge ideal of a cograph in terms of its number of vertices. It is proven that this quantity is at most $\frac{2}{3}n$. The methods of that work are mostly combinatorial, and lean heavily on the fact that cographs can be built in a simple recursive way. While our methods are not as well suited to proving these kinds of precise numerical bounds, Corollary \ref{ComAlgCor} does imply a uniformity of the shape of syzygies which can appear across all cographs, and across much more general families of ideals. One can think of the proofs of these results resulting from a kind of \emph{categorification} of the recursive methods of \cite{kahle2019binomial}.

\subsection{Graph complexes of cographs}

In this section we begin to discuss a number of topological implications of our main technical Theorem \ref{mainTech}. Our first result is an immediate consequence of our main technical theorem.

\begin{corollary}\label{cor:graphCom}
    If $\Delta_G$ is any vertex monotone graph complex, and $i \geq 0$ is any integer, then the $\Co$-module
    \[
    G \mapsto H_i(\Delta_G;\Z)
    \]
    is finitely generated. In particular, for every vertex monotone graph complex $\Delta_\bullet$, there exists an integer $d_{i,\Delta}$ such that the exponent of torsion appearing in $H_i(\Delta_G;\Z)$, across all cographs $G$, must divide $d_{i,\Delta}$.
\end{corollary}

\begin{proof}
    The approach here is very common (see \cite{miyata}, for instance). The idea is that we may view $H_i(\Delta_G,\Z)$ as a subquotient of the wedge power $\bigwedge^i\Z V_G$. By our assumption, the full embeddings of the category $\Co$ turn the assignment 
    \[
    G \mapsto \bigwedge^i\Z V_G
    \]
    into a $\Co$-module. Certainly $G \mapsto \Z V_G$ is itself finitely generated as a comodule, so to show that $G \mapsto \bigwedge^i\Z V_G$ is finitely generated, it suffices to show that tensor powers of the vertex $\Co$-module is finitely generated. Indeed, it isn't hard to see that the $i$-th tensor power of the vertex module is generated by cographs of at most $i$ vertices without edges. The Noetherianity theorem \ref{mainTech} now implies our result.
\end{proof}

\begin{remark}
    One might notice from the above proof that an identical statement can be proven for complexes built out of the edges of the graph. Indeed, in \cite{jonsson2005simplicial,miyata} consider complexes built out of collections of edges, which they call \emph{edge} monotone. Examples of such complexes include matching complexes, which have been a source of extensive research for many years \cite{ShareshWachs,WachsSurvey}. The reason we have chosen to frame our discussion in terms of specifically \emph{vertex} monotone complexes is because of the fact that Theorem \ref{mainTech} fails if one tries to set $|\bullet|$ to be the edge functor. In other words, while we only used Noetherianity of $\Co$-modules to prove the above statement, the follow-up results in this section will require the full strength of Theorem \ref{mainTech}, and will therefore not be applicable to the edge case. 
\end{remark}

\begin{example}
    Of course, the results of Corollary \ref{cor:graphCom} mean very little if these homology groups are just always eventually zero! To show that this isn't the case, consider the vertex-monotone graph complex which associates to every cograph itself viewed as a simplicial complex. This can be equivalently viewed as the 1-skeleton of the clique complex.

    In this case the only non-zero homology group is $H_1$, which is free abelian of rank $|E_G|-|V_G|+1$. In fact, we can say something stronger. Because every induced cycle in a cograph is either a triangle or a square, this particular $\Co$-module can be seen to be generated by the triangle and the square.
\end{example}

It is a well-studied phenomenon that the homology groups of the matching complexes of complete graphs and complete bipartite graphs can contain non-trivial torsion \cite{ShareshWachs}. In particular, it is known that these kinds of vertex/edge monotone complexes can admit torsion in their homology groups in the case of cographs. Moreover, per the conclusions of \cite{ShareshWachs}, it seems as if this torsion is exceptional difficult to find and analyize in these circumstances. Therefore, the consequence described in Corollary \ref{cor:graphCom} does contain non-trivial information.
 
One limitation of Corollary \ref{cor:graphCom} is that it only says anything in the cases wherein the homological index is fixed. In other words, it tells us comparatively little about things such as Euler characteristics or Poincare polynomials, which would require knowledge of the Betti numbers in homological indices relative to the order of $G$. In order to bridge this gap, we will require the theorems of the previous section.

\textbf{For the remainder of this section, we assume that $k$ is a field.}

\begin{definition}
    Let $\Delta$ denote an abstract simplicial complex on a vertex set $V$. Then the associated \textbf{Stanley-Reisner ideal} $I_\Delta$ is defined by
    \[
    I_\Delta = (x_{v_1}x_{v_2}\cdots x_{v_n} \mid \{v_1,\ldots,v_k\} \notin \Delta\} \subseteq k[x_v \mid v \in V] = k[V]
    \]
    
    Conversely, to any ideal of $k[V]$ which is generated by squarefree monomials, one can associate a simplicial complex.
\end{definition}

The correspondence between a simplicial complex and its Stanley-Reisner ideal is extremely robust in a number of different ways. For instance, it isn't hard to see that the topological dimension of the simplicial complex will always be one less than the Krull dimension of the quotient ring $k[V]/I_\Delta$. In fact, more generally, the Hilbert series of the (coarsely) graded module  $k[V]/I_\Delta$ will always have a rational expression of the form
\[
\frac{\sum_i h_it^i}{(1-t)^d},
\]
where $d$ is one higher than the dimension of $\Delta$, and the $h_i$ form the $h$-vector of $\Delta$ \cite{miller2005combinatorial}.

For what we need, however, we will want to consider connections between the ideal $I_\Delta$, and the topological (co)homology of the simplicial complex $\Delta$. This is best illustrated by the celebrated formula of Hochster.

\begin{theorem}[Hochster's formula, Corollary 5.12 \cite{miller2005combinatorial}]
	Let $\Delta$ be a simplicial complex with corresponding Stanley-Reisner ideal $I_\Delta$, which we consider multi-graded by the vertex grading. The nonzero Betti numbers of $ I_\Delta $ lie only in square free multi-degrees. Namely, those multi-degrees $\mathbf{a}$ whose every entry is at most 1, and therefore correspond to a subset of $V_\Delta$. Furthermore,
	\[
		\beta^{i , \mathbf{a}}(I_\Delta) = \dim_{k} \tilde{H}^{|\mathbf{a}| - i - 2}( \Delta|_{\mathbf{a}} ; k)
	\] 
	where $\tilde{H}$ is reduced (co)homology , $|\mathbf{a}|$ is the sum of the entries of $\mathbf{a}$ (or the size of the corresponding set of vertices), and $\Delta|_{\mathbf{a}} = \{\tau \in \Delta\mid \tau \subseteq \mathbf{a}\}$. 
\end{theorem}

Returning to our context, and applying the theorems of the previous section, we recover the following.

\begin{corollary}\label{finite_betti_cohom}
	Fix $i \geq 0 $ and a vertex monotone graph complex $\Delta_G$. There exists an integer $ n_{i} $, depending only on $i$, such that for any cograph $ G $, if $ \mathbf{a} \subseteq V_G$ with $ |\mathbf{a}| > n_{i} $, then 
	\[
		\dim_{k} \tilde{H}^{|\mathbf{a}| - i - 2}(\Delta_G |_{\mathbf{a}} ; k) = \dim_{k} \tilde{H}_{|\mathbf{a}| - i - 2}(\Delta_G |_{\mathbf{a}} ; k) = 0.
	\]

    In particular, for any vertex monotone graph complex $\Delta_G$ and any $i \geq 0$, there are at most finitely many cographs for which $\dim_{k} \tilde{H}_{|V_G| - i - 2}(\Delta_G ; k)$ is non-zero
\end{corollary}

\subsection{Variations on configuration space}

One of the original applications of the ``categorified" recursive classification of cographs was to show that graph configuration spaces of cographs had very regular behaviors in their homology groups \cite{KR}. This this section, we use the expanded methods of this work to prove similar statements for a number of variations on that theme.

\subsubsection{Anchored configuration spaces of cographs}

As our first variation, we considered the so-called anchored configuration spaces \cite{kozlov2023stirling,kozlov2022configuration,kozlov2024homology}. These spaces have seen uses in applied topology problems related to resource distribution and logistics \cite{kozlov2023stirling,hoekstra2025cup}.

\begin{definition}
 Let $G$ be a graph, considered as a 1-dimensional simplicial complex, and let $K$ be a finite set of vertices of $G$. Then the \textbf{$n$-pointed anchored configuration space of $G$ relative to $K$}, $\Sigma(G,K,n)$ is the subcomplex of $G^n$ comprised of all points $(x_1,\ldots,x_n)$ such that every element of $K$ appears at least once among the $x_i$.
\end{definition}

\begin{remark}
Note that, unlike traditional configuration spaces, points are allowed to collide in these anchored spaces. Indeed, otherwise the above would be equivalent to a typical configuration space of fewer points on the graph with $K$ removed. In this sense, they are similar to the graph configuration spaces with ``sinks" studied in \cite{chettih2018homology,ramos2019configuration}, which allow collisions at the vertex set of the underlying graph, though nowhere else.
\end{remark}

\begin{remark}
    In \cite{kozlov2024homology}, Kozlov introduced a slight variation of the above, wherein one only requires a fixed number of the elements of $K$ be always present among the configured points. These generalized anchored configuration spaces can also be treated with the methods of this work. We opt to only consider the usual anchored spaces for ease of exposition and for brevity.
\end{remark}

In order to apply our categorical methods to these kinds of anchored configuration spaces, we will need to slightly generalize the category we are looking at to incorporate the data of the set $K$.

\begin{definition}
    Let $r \geq 0$ be an integer. We will use $\Co_r$ to denote the category whose objects are pairs $(G,K)$ of a cograph along with a set $K$ of vertices of $G$ with $|K| = r$. The morphisms, $\phi:(G,K) \rightarrow (G',K')$, of this category are full embeddings $\phi:G \rightarrow G'$ with the added property that $\phi(K) = K'$.\\

    We will use $\Cot_r$ to denote the category whose objects are cotrees, with the extra label $K$, which is applied to exactly $r$ leaves. The morphisms of this category are label preserving homeomorphic embeddings. We can similarly define a category of planar cotrees equipped with the data of $K$.

    Defining $|(G,K)|$ to be the vertex set of $G$ turns $\Co_r$ into an $\FI$-concrete category as before. It is immediate that the category $\Cot_r$ is equivalent to $\Co_r$, and that this equivalence extends to an equivalence of $\FI$-concrete categories. It is also clear that the forgetful functor from planar cotrees with the data of $K$ to cotrees with the data of $K$ is essentially surjective and concrete.
\end{definition}

All of the above leads to the following very natural generalization of Theorem \ref{mainTech}.

\begin{theorem}\label{thm:anchorPoly}
    Let $k$ be a Noetherian ring, $r\geq 0$ an integer. If $V$ is a finitely generated module over $\Co_r$, then all of its submodules are also finitely generated. More generally, writing $A_{|\bullet|}$ for the polynomial ring over $\Co_r$, if $M$ is a finitely generated module over $A_{|\bullet|}$, then all submodules of $M$ are also finitely generated. 
\end{theorem}

\begin{proof}
It will suffice to show that the category of planar cotrees with the data of $K$ forms a strongly G-category. The proof here is identical to the proof that $\PCot$ was a strongly G-category, with the added dimension that our leaves might have the extra label $K$. Because there is only a single additional label that is always applied to a fixed (i.e. $r$) number of leaves, these data can be incorporated without breaking the necessary well-quasi-order properties that were used to prove that the category is a strongly G-category.
\end{proof}

We recover the following as a consequence:

\begin{corollary}\label{cor:anchor}
    Let $i,r,n \geq 0$ be fixed integers. Then the assignment,
    \[
    G \mapsto H_i(\Sigma(G,K,n);\Z)
    \]
    defines a finitely generated module over $\Co_r$. In particular,
    \begin{enumerate}
        \item For any $i,r,n \geq 0$, there exists a finite list of pointed cographs $\{(G_j,K_j)\}_{j = 1}^l$ with $|K_j| = r$, such that for any pointed cograph $(G,K)$ with $|K| = r$, the homology group $H_i(\Sigma(G,K,n);\Z)$ is generated by the homology classes of $H_i(\Sigma(G_j,K_j,n);\Z)$ pushed forward along full embeddings $(G_j,K_j)\hookrightarrow (G,K)$;
        \item For any $i,r,n \geq 0$, there exists an integer $d_{i,r,n} \geq 0$ such that for any pointed cograph $(G,K)$ with $|K| = r$, the exponent of the group $H_i(\Sigma(G,K,n);\Z)$ divides $d_{i,r,n}$;
        \item For any $i,r,n \geq 0$, and any field $k$, the function,
        \[
        m \mapsto \dim_k(H_i(\Sigma(K_m,[r],n);k))
        \]
        agrees with a polynomial for all $n \gg 0$.
    \end{enumerate}
\end{corollary}

\begin{proof}
    Let $C_{i,r,n}(G,K)$ denote the free abelian group of (cubical) $i$-chains of $\Sigma(G,K,n)$. It will suffice to show that this is finitely generated as a $\Co_r$-module. Note that this $\Co_r$-module is a submodule of the cubical $i$-chains of the entire product $G^n$. This latter module can be seen to be finitely generated by the collection of pointed graphs with at most $n-i+2i = n+i$ vertices.

    For the third statement, first note that the homeomorphism type of the anchored configuration space on $n$ points on any complete graph $K_m$ only depends on the number of points in the configuration space, and the size of the anchor set. Let $\FI_{+r}$ be the category whose objects are sets of the form $[r+n] = \{1,\ldots,r,r+1,\ldots,r+n\}$, and whose morphisms are injections which fix $1,\ldots,r$. Then for any $r,n \geq 0$,
    \[
    [m] \mapsto H_i(\Sigma(K_m,[r],n);k)
    \]
    is seen to be a module over $\FI_{+r}$. This module will be finitely finitely generated by the prior parts of this corollary, by restricting to the full subcategory comprised of only the complete graph and anchor set equal to $[r]$. To conclude the proof, we observe that $\FI_{+r}$ is equivalent to $\FI$, and apply the polynomial dimension growth theorem for $\FI$-modules \cite{church2014fi}.
\end{proof}

The question of integral torsion in the more typical configuration spaces of graphs is still a hotly researched one. For example, while it has been known that the first integral homology groups of \emph{unordered} configurations can have torsion (and this torsion actually detects planarity of the graph) for at least a decade \cite{ko2012characteristics}, there has still not been confirmation whether torsion can appear in any other homology group. In fact, arguably the most significant outstanding question in the theory of graph configuration spaces is whether or not it is possible for there to ever be \emph{odd} torsion in their homology groups. In the case of \emph{ordered} configurations, it is conjectured that there cannot be torsion, but it is not known in virtually all cases (See \cite{hainaut2025representation} on the most up to date theorems about torsion in ordered configuration spaces of graphs).

To date, it is not known whether or not anchored configuration spaces can admit integral torsion in their homology groups. Just as with standard configuration spaces, it was shown that the homology groups of anchored trees are always torsion free \cites{chettih2018homology,kozlov2023stirling}. It was later shown that the same is true in the cases where the graph is a cycle \cite{kozlov2022configuration}. Outside of trees and the cycle, however, there have not been nearly enough computations completed in the homology groups of anchored configuration spaces to say one way other the other whether torsion can appear. Despite this, the above theorem implies that the kinds of torsion that can appear, at least in the cases of cographs, must be universally bounded for each fixed $r,i,n \geq 0$.

In \cite{kozlov2023stirling}, where the study of anchored configuration spaces was initiated, Kozlov provides a classification of the homotopy types of anchored configuration spaces of trees. In that work it is noted that the homotopy type of the space depended on very little information, the number of points being configured and the number of vertices in the tree. The first part of Corollary \ref{cor:anchor} suggests that a similar phenomenon can likely be described for cographs, atleast at the level of the homology groups. Namely, there can only be a finite amount of data that any homology group can possibly depend on.

While the proof of \ref{cor:anchor} only requires a Noetherianity statement at the level of a particular category's representations, Theorem \ref{thm:anchorPoly} gives us the strictly stronger Noetherianity statement for modules over a categorical polynomial ring. It therefore becomes natural to ask whether this extra structure can be used to more deeply study the homology groups of anchored configuration spaces. Indeed, because our spaces allow for collisions, one has natural maps associated to each vertex of a given graph $G$,

\[
\Sigma(G,K,n) \stackrel{\cdot x_v}{\rightarrow} \Sigma(G,K,n+1),
\]

given by adding a particle to the specified vertex. The issue with this, however, is that there is no natural way to account for the order of the points. In other words, the proper action here would not be by the polynomial ring over $\Co$, but rather some kind of \emph{non-commutative} polynomial ring over $\Co$. Something similar to this was accomplished for graphical configuration spaces with sinks in \cite{ramos2019configuration}, and for configurations of disks in an infinite strip in \cite{alpert2024configuration,wawrykow2024representation}.

Of course, one loses this need for order preservation if there isn't an ordering on the points to begin with! In other words, writing
\[
\sigma(G,K,n) := \Sigma(G,K,n)/S_n,
\]
the assignment,
\[
(G,K) \mapsto H_i(\sigma(G,K,n))
\]
does define a module over the polynomial ring over $\Co_r$. Note that, unlike in the traditional configuration space setting, the action of the symmetric group by permuting points is not properly discontinuous (in fact, it isn't even free). This suggests that the spaces $\sigma(G,K,n)$ are slightly more topologically fraught. Indeed, to the knowledge of the authors, there has yet to be any serious investigation into these spaces.

\subsubsection{Cographical configuration spaces}

The second variation on classical configuration spaces that we will concern ourselves with are graphical hyperplane arrangements.

\begin{definition}
    Given a graph $G$ with vertex set $V$ and edge set $E$, we define the associated \textbf{graphical hyperplane arrangement} \footnote{The phrase ``hyperplane arrangement" in this context will always refer to an arrangement which is complex, and central. In other words, we will assume that our hyperplanes all intersect at the origin.} as follows. The ambient space of our arrangement is the vector space $\mathcal{V}_G = \C^{V_G}$, whereas the hyperplanes are those defined by the equations $x_v = x_w$, whenever $v \sim w$ in $G$.
\end{definition}

\begin{example}
    If $G = K_n$ is the complete graph, then this hyperplane arrangement is the famous \textbf{braid arrangement}. On the other hand if $G = K_{a,b}$ is the complete bipartite graph, what results is sometimes referred to as a \textbf{colored braid arrangment} \cite{farb2019coincidences}.
\end{example}

Given a complex hyperplane arrangement there are a number of ways to assign meaningful geometry (see \cite{huh2022combinatorics} for a very recent survey on the matter). In this paper we will largely stick to the first, and most natural construction, the complementary space.

\begin{definition}
    For a given the graphical hyperplane arrangement $\mathcal{A}_G$, we define its \textbf{complementary space} to be the complex variety defined by
    \[
    \mathcal{M}(G) = \{(x_v)_{ v \in V_G} \mid x_v \neq x_w \text{ if $v \sim w$}\}
    \]
\end{definition}

\begin{example}
    The complementary space of the braid arrangement is configuration space on $n$ points in the plane. The complementary space of the colored braid arrangements are the colored configuration spaces considered in \cite{farb2019coincidences}.

    Configuration space of the plane has been studied since the work of Arnold \cite{arnold2013cohomology}. These spaces are also the motivating example behind \textbf{representation stability}, and the theory of \textbf{FI-modules} \cites{church2013representation,churchFIMod}. One of the primary results of this paper will show that the complements of hyperplane arrangements show finitely generated behaviors across \emph{all} cographs, not just the complete graph and complete bipartite graph cases which were known from those works.
\end{example}

Important for us will be the following celebrated theorem of Orlik and Solomon from \cite{orlik1980combinatorics}.

\begin{definition}
    Let $\mathcal{A}_G$ denote the graphical hyperplane arrangement associated to a graph $G$, and let $\mathcal{E}_G$ denote the exterior algebra of the vector space whose basis is in bijection with the hyperplanes of $\mathcal{A}_G$, $<e_H>$. For any collection of hyperplanes in $\mathcal{A}_G$ $H_1,\ldots,H_n$, we set
    \[
    e_S = e_{H_1} \wedge e_{H_2} \wedge \ldots \wedge e_{H_n}.
    \]
    We say that a collection of hyperplanes in $\mathcal{A}_G$ is \textbf{dependent} if the normal vectors to these hyperplanes are linearly dependent in the usual sense.
    
    The algebra $\mathcal{E}_G$ can be given the structure of a d.g. algebra by defining the boundary map
    \[
    \partial(e_S) = \sum_i (-1)^i e_{S - \{H_i\}}
    \]

    The Orlik-Solomon algebra $OS(G)$ is defined to be the quotient of $\mathcal{E}_G$ by the ideal generated by $\partial e_S$, where $S$ is a dependent collection of hyperplanes.
\end{definition}

\begin{theorem}\label{thm:OSalg}
Let $\mathcal{A}(G)$ be a graphical hyperplane arrangement. Then there is an isomorphism of graded algebras
\[
H^\star(\mathcal{M}(G);\Z) \cong OS(G)
\]
\end{theorem}

\begin{remark}
    One other topic in the geometry of matroids that has seen approaches using representations of categories is the \textbf{Kazhdan–Lusztig} polynomial of the matroid \cites{elias2016kazhdan,proudfoot2017configuration}. In \cite{proudfoot2017configuration} it is shown that these coefficients have a form of finite generation with respect to the category of finite sets and surjections in the case of the braid matroid. Being that the braid matroid is a special case of a graphical matroid of a cograph, it is natural to ask whether those results can be lifted to this level of generality. Unfortunately this is not the case. The primary machinery that allows the work of \cite{proudfoot2017configuration} to function is a spectral sequence relating the Kazhdan–Lusztig polynomial of a hyperplane arrangement to that of a contraction. In the case of the complete graph, any surjection $[a] \twoheadrightarrow [b]$ induces an edge contraction turning the complete graph $K_b$ into $K_a$ (up to removal of multi-edges and loops). In the case of general cographs, however, there is no clear functoral way to realize an induced subgraph as an edge contraction of a larger cograph. 
    
    In \cite{flynn2024representation}, Flynn, the third author, and Young illustrated how the results of \cite{proudfoot2017configuration} on the Kazhdan-Lusztig coefficients of the braid arrangement seem to be quite difficult to generalize.
\end{remark}

All of the above background will now allow us to prove the following theorem, which generalizes and expands upon the seminal representation stability results of the braid and colored braid arrangements \cite{church2013representation,farb2019coincidences}.

\begin{corollary}
    Let $\mathcal{A}(G)$ denote the graphical hyperplane arrangement with complement space $\mathcal{M}(G)$. Then for any fixed $i \geq 0$, the assignment,
    \[
    G \mapsto H^i(M(G);\Z)
    \]
   defines a finitely generated $\Co$-module.
\end{corollary}

\begin{proof}
    Theorem \ref{thm:OSalg} tells us that these cohomologies are natural subquotients of wedge powers of the linearization of the edge set. We have already seen that these wedge-powers must all be finitely generated, and therefore the same can be said of their quotients.
\end{proof}

\bibliography{bibliography}
\bibliographystyle{amsalpha}

\end{document}